\documentclass[12pt,twoside]{article}
\usepackage[english]{babel}
\usepackage[latin1]{inputenc}
\usepackage{amsmath}
\usepackage{amsthm}
\usepackage{amssymb,amsfonts}
\usepackage{comment}
\usepackage{tikz-cd}
\usepackage{tikz}
 \usetikzlibrary {shapes.geometric} 
\usepackage{orcidlink}
\usepackage{stmaryrd}
\usepackage{caption}
\usepackage{graphicx}                   
\usepackage{hyperref}

\newcommand{\CP}[1]{\mathbb{C}P^{#1}}      

\begin{document}
\raggedbottom
\pagestyle{myheadings}
\title
{\textsc{Rank three representations of Painlev\'e systems: I. Wild character varieties}}

\author{by \textsc{Mikl\'os Eper}\footnote{\textsl{Department of Algebra and Geometry, Institute of Mathematics, Faculty of Natural Sciences, Budapest University of Technology and
Economics, M\H uegyetem rkp. 3., Budapest H-1111, Hungary}, e-mail: \href{mailto:epermiklos@gmail.com}{epermiklos@gmail.com}} and \textsc{Szil\'ard Szab\'o}\footnote{Institute of Mathematics, Faculty of Science, E\"otv\"os Lor\'and University, P\'azm\'any P\'eter s\'et\'any 1/C, Budapest, Hungary, H-1117; Alfr\'ed R\'enyi Institute of Mathematics,
Re\'altanoda utca 13-15., Budapest 1053, Hungary, e-mails:\href{mailto:szilard.szabo@ttk.elte.hu}{szilard.szabo@ttk.elte.hu} and \href{mailto:szabo.szilard@renyi.hu}{szabo.szilard@renyi.hu}}
\orcidlink{https://orcid.org/my-orcid?orcid=0009-0008-7435-2021}}
\maketitle


\begin{abstract}
We give explicit cubic equations for the wild character varieties corresponding to the rank $3$ representations of Painlev\'e equations, and compare them to the ones of their classical rank $2$ representations. 
\end{abstract}

\newtheorem{theorem}{Theorem}[section]
\newtheorem{corollary}[theorem]{Corollary}
\newtheorem{conjecture}{Conjecture}[section]
\newtheorem{lemma}[theorem]{Lemma}
\newtheorem{exmple}[theorem]{Example}
\newtheorem{defn}[theorem]{Definition}
\newtheorem{prop}[theorem]{Proposition}
\newtheorem{rmrk}[theorem]{Remark}
\newtheorem{claim}[theorem]{Claim}

\newenvironment{definition}{\begin{defn}\normalfont}{\end{defn}}
\newenvironment{remark}{\begin{rmrk}\normalfont}{\end{rmrk}}
\newenvironment{example}{\begin{example}\normalfont}{\end{example}}
\newenvironment{acknowledgement}{{\bf Acknowledgement:}}

\newcommand\restr[2]{{
  \left.\kern-\nulldelimiterspace 
  #1 
  \vphantom{\big|} 
  \right|_{#2} 
  }}


\section{Introduction}

Since their discovery at the beginning of the 20th century, the Painlev\'e equations have attracted a lot of attention, and found a plethora of applications ranging from the theory of random matrices to integrable systems and quantum field theory. 
A classical result in the field due to Fuchs and Garnier (see Jimbo and Miwa~\cite[Appendix~C]{JM}) is that the Painlev\'e equations all admit descriptions as isomonodromy equations of rank $2$ connections with (regular and irregular) singularities on the Riemann sphere. 
Meanwhile, Flaschka and Newell~\cite{FN} found a different such rank $2$ isomonodromy equation description of the second Painlev\'e equation Painlev\'e II.
This development pointed out that the Painlev\'e equations may admit several isomonodromic interpretations. 
Such alternative interpretations (known in the literature as Lax pair representations) are important in order to recognize the relevance of the Painlev\'e equations for further potential applications. 
Another Lax pair representation, this time in terms of rank $3$ singular connections, for the Painlev\'e VI equation was found by Harnad~\cite[Section~3c]{Har}. 
Subsequently, it was proved by Mazzocco~\cite{Maz} that Harnad's duality can in fact be considered as an instance of Fourier--Laplace transformation of $\mathcal{D}$-modules. 
Later, Joshi, Kitaev and Treharne~\cite{JKTI},~\cite{JKTII} gave a list of rank $3$ Lax (or, using their terminology, Fuchs--Garnier) pair representations of all the Painlev\'e equations.\footnote{We note, however, that for Painlev\'e III they did not quite find a connection, because the differential term is multiplied by a singular matrix.} 
More generally, in~\cite[Sections~9,~10]{Boa5}, Boalch described a way to produce several Lax representations of a class of wild character varieties, namely the ones that are generalized multiplicative quiver varieties with a "supernova graph" as underlying graph. 
By assumption, these cases must have unramified irregular type. 
Among the Painlev\'e cases studied here, this applies to IVb and VI. 
On the other hand,~\cite{DH} covers the cases with a single slope $k$ at each irregular singular point (see~\cite[Assumption~4.1]{DH}), including the case JKT I worked out here (with $k=5/3$). 
Finally, in~\cite[Section~4]{Dou} Dou\c{c}ot upgraded Boalch' quiver variety interpretation to describe various Lax representations of all the Painlev\'e spaces, containing some (but not all) of the JKT examples. 
Notice that, contrary to the present paper, Dou\c{c}ot works with moduli spaces of connections, i.e. on the de Rham side. 
He expects that Fourier--Laplace transformation gives isomorphisms between the related nonabelian Hodge spaces, i.e. in all complex structures of the hyperK\"ahler family. 
In a sequel of this paper~\cite{ESz_Fourier}, we will establish this isomorphism result in the cases under consideration here with respect to the de Rham complex structure. 
In the third and last part of this sequence of papers~\cite{ESz_Dolbeault}, we establish a similar isomorphism in the Dolbeault complex structure, with in addition the isomorphism of the holomorphic symplectic structure too, thus closing the circle and proving hyperK\"ahler isometry between the rank 2 and rank 3 representations of the Painlev\'e spaces. 
We expect that the approach of this sequence of papers should generalize to establishing similar isometry results in higher generality, too. 

The aim of the present article is to study the wild character varieties obtained from these Joshi--Kitaev--Treharne (JKT) irregular connections upon application of the Riemann--Hilbert--Birkhoff correspondence, and compare them with the character varieties belonging to the corresponding rank 2 Jimbo--Miwa (JM) connections form \cite{JM}. 
In our prior work~\cite{ESz}, we used trace coordinates to identify the character varieties in the logarithmic case as a family of affine cubic surfaces. 
Here, we use the coefficients of the Stokes matrices to give a similar description of the wild character varieties, again as affine cubic surfaces. 
On the other hand, the wild character varieties corresponding to the JM representations of the Painlev\'e cases have been analyzed in detail by van der Put and Saito~\cite{vPS}. 
In particular, they presented these spaces as affine cubic surfaces. 
\begin{theorem}
\label{Thm:thm1}
    Under suitable correspondence of parameters in each Painlev\'e case, the affine cubic surface describing the JKT wild character variety is isomorphic to the corresponding (rank 2) Jimbo--Miwa (JM) character variety. 
\end{theorem}
For the concrete correspondence with the parameters of the JM wild character varieties, see Remarks~\ref{rem:JKTVI}--\ref{rem:JKTII}. 
As opposed to the logarithmic case~\cite{ESz}, the equations we find here possess a single degree $3$ term, of the form $XYZ$. 

The fundamental reason behind Theorem \ref{Thm:thm1} is that Fourier--Laplace transformation establishes an isometry between the six de Rham moduli spaces of rank three irregular connections belonging to the JKT systems and the corresponding moduli spaces of rank two irregular connections belonging to JM systems. 
The details of this correspondence will be discussed in a sequel to this paper~\cite{ESz_Fourier}. 
We note that a similar relationship was shown to hold for connections with a single untwisted irregular singularity of Poincar\'e rank $1$, namely the Stokes data can be derived from the monodromy data of the Fourier--Laplace transformed Fuchsian system~\cite{BJL}.
This case was also discussed in~\cite[Chapter~XII]{Mal} using the nearby and vanishing cycle functors, which was later explained in~\cite[Theorem~5.4]{DHMS} in a topological way, using enhanced  Fourier--Sato transform. 

Character varieties of (irregular) curves are of interest and admit applications throughout mathematics and mathematical physics. 
There are also quite interesting recent conjectures about them, for instance the Geometric P=W conjecture. 
Even though many theoretical results are known in the irregular case, relatively few concrete examples exist in the literature, especially in the higher (than $2$) rank twisted case. 
This study illustrates how the general theory of~\cite{Boa5} works in practice, by providing a detailed description of a family of rank $3$ examples that is known to be related to the well-known Painlev\'e family. 

In Section~\ref{sec:prelim} we list a few fundamental results needed in the sequel. 
In Section~\ref{sec:JKT} we spell out the local forms of the connections that we deal with. 
Finally, in Section~\ref{sec:Betti} we get, case by case, equations for the wild character varieties.

\section{Preliminaries}\label{sec:prelim}

\subsection{Meromorphic connections} \label{sec:connection}

Let $X = \mathbb{C} P^1$ be the complex projective line and let $\mathcal{D}_X$ be the sheaf of analytical differential operators on $X$. 
It is the sheaf of noncommutative $\mathcal{O}_X$-algebras that is generated by $\partial_z$ with its standard chart $z$ at $[0\colon 1]\in \mathbb{C} P^1$, subject to the relation $[\partial_z, z] = 1$. 
It has a similar description over $\operatorname{Spec}\mathbb{C}[w]$ with the standard chart $w = z^{-1}$ at $[1\colon 0]\in \mathbb{C} P^1$, with the relationship $w^{-1}\partial_w = - z^{-1}\partial_z$ between its generators. 
Similarly, for any variable $t$ one can also define the noncommutative algebra $\mathbb{C}[t] \langle \partial_t \rangle$ of polynomial differential operators on $X$. 
A left module $\mathbb{M}$ over $\mathcal{D}_X$ is said to be holonomic if it is finitely generated and torsion. 
In the sequel, $\mathbb{M}$ will always denote a holonomic left $\mathcal{D}_X$-module. 
The rank of $\mathbb{M}$ is then defined as 
\[
    \operatorname{rk} \mathbb{M} = \dim_{\mathbb{C}(z)} \mathbb{C}(z) \otimes_{\mathbb{C}[z]} \mathbb{M} .
\]
Moreover, there exists a maximal Zariski open subset $U\subset X$ such that $\mathbb{M}|_U$ is finitely generated over $\mathcal{O}_U$. 
The set of singular points of $\mathbb{M}$ is then defined to be $X\setminus U$, and denoted by $C = \operatorname{Sing}(\mathbb{M})$. We denote by $j$ the inclusion $U\to X$. 

\subsubsection{Formal theory}

We will focus on the formal structure of $\mathbb{M}$ at $\operatorname{Sing}(\mathbb{M})$. 
For any $c\in \operatorname{Sing}(\mathbb{M})$ we pick a local holomorphic coordinate $t$ of $X$. 
For $c\in \operatorname{Spec}\mathbb{C}[z]$, we may take $t = z-c$, and for $c=\infty$, we take $t=w$. 
A special role is played by regular singular modules, which are by definition the ones that admit a fundamental system of at most polynomial growth in $|t|$ as $t\to 0$ within a sector of finite opening. 
We next consider 
\[
    \mathbb{M}_c =  \mathbb{C} \llbracket t \rrbracket \langle \partial_t \rangle \otimes_{\mathbb{C}[t] \langle \partial_t \rangle} \mathbb{M} .
\]
The Hukuhara--Levelt--Turrittin theorem~\cite{Huk},~\cite{Lev},~\cite{Tur}  then states that there exists some positive integer $d$ and a finite Galois extension 
\begin{equation}\label{eq:extension}
    K = \mathbb{C} (\!( u , t )\!)/(u^d - t)     
\end{equation}
of $\mathbb{C} (\!( t )\!)$  such that 
\begin{equation}\label{eq:HLT}
    K\otimes_{\mathbb{C}[t]} \mathbb{M} \cong \bigoplus_{q\in K/\mathbb{C}\llbracket u \rrbracket} 
(\mathbb{C}\llbracket u \rrbracket ,\operatorname{d}+\operatorname{d}\! q) 
\otimes_{\mathbb{C}\llbracket u \rrbracket} \mathcal{F}_{q} .
\end{equation}
Here $\operatorname{d}$ stands to denote the trivial connection and the exterior differentiation operator with respect to $u$ and $\mathcal{F}_{q}$ is a free $\mathbb{C}\llbracket t \rrbracket$-module with regular singular connection. 
For all but finitely many values of $q$ the rank of $\mathcal{F}_{q}$ is $0$. 
The values $q$ for which $\mathcal{F}_{q}\neq 0$ are called the \emph{exponential factors} of $\mathbb{M}_c$. 
The decomposition is unique up to permutations, gauge transformations, and further field extensions. 

We will be interested in the special cases of this general theory, detailed in Section \ref{sec:JKT}.

\subsection{Wild character varieties} \label{wild character varieties}
In this section we define wild character varieties over $X=\CP1$ with structure group $G=\operatorname{GL}(3,\mathbb{C})$. 
Due to the presence of irregular singularities, these involve data called Stokes matrices. 
The discovery of the Stokes phenomenon goes back to the study of the Airy equation by Stokes, and was later studied by many authors, see~\cite{Was}. 
Note that the assumption $g=0$ causes simplifications in some of the formulas. 
We will freely use the language of quasi-Hamiltonian geometry~\cite{AMM}. 

\subsubsection{Stokes local systems -- untwisted case}
We fix a maximal torus $T\subset G$ with Lie algebra $\mathfrak{t}$. 
An \emph{irregular curve}~\cite[Definition~8.1]{Boa4} is the data of a compact Riemann surface $X$ (that we will always take to be $\CP1$), a finite set of points $\mathbf{a} = \{ a_i \}_{i=1}^m$, $a_i \in X$ and for all $i$ an \emph{irregular type} 
\begin{equation}
\label{eq:irreg}
    Q_i = \frac{A_{r_i}}{z_i^{r_i}} + \cdots + \frac{A_{1}}{z_i}
\end{equation}
for some holomorphic coordinate function $z_i$ centered at $a_i$ and $A_j\in \mathfrak{t}$. This is essentially \eqref{eq:HLT} expanded in matrix form.
Then $Q_i$ singles out a subgroup 
\[
    H_i = \{ g\in G \vert \, \operatorname{Ad}_g(A_j) = A_j \; \mbox{for all}\; j\geq 1 \} \subset G
\]
called its stabilizer.
Let $\mathcal{R}\subset \mathfrak{t}^{\vee}$ be the set of roots of the Lie algebra $\mathfrak{g}$ of $G$ with respect to $\mathfrak{t}$, and for any $\alpha\in\mathcal{R}$ let us denote by $\mathfrak{g}_{\alpha}$ the corresponding root space. 
Consider the real oriented blow-up $\widetilde{X}$ of $X$ at $a_i$. 
Let us denote by $\partial_i$ the fiber of $\widetilde{X}$ over $a_i$, which is a boundary circle. 
Elements of $\partial_i$ are called directions at $a_i$, and are denoted by $d$. 
A direction $d$ is called a \emph{singular direction} supported by $\alpha\in\mathcal{R}$ if it is the tangent direction to one of the components of the curve 
\begin{equation}\label{eq:singular_direction}
    \operatorname{Im}(\alpha \circ Q (z_i )) = 0, \quad \operatorname{Re}(\alpha \circ Q (z_i )) < 0.     
\end{equation}
It is clear that if $\alpha \circ Q$ has a pole of order $k$ then there exist exactly $k$ singular directions supported by $\alpha$ (and the $k$ singular directions supported by $-\alpha$ are their opposites). 
For any singular direction $d$ let us denote by $\mathcal{R}(d)$ the set of roots $\alpha$ supporting the singular direction $d$. 
For all but finitely many $d\in \partial_i$, we have $\mathcal{R}(d)=\varnothing$. 
Given $d\in \partial_i$ singular, we define the unipotent subgroup $\mathbb{S}\!\operatorname{to}_d (Q_i) \subset G$ as the closed subgroup corresponding to the Lie subalgebra spanned by $\mathfrak{g}_{\alpha}$ as $\alpha$ ranges over $\mathcal{R}(d)$. 
If $d\in \partial_i$ is not singular, then $\mathbb{S}\!\operatorname{to}_d (Q_i)=\{ \operatorname{I} \}$. 
Finally, we set 
\[
    \mathbb{S}\!\operatorname{to}(Q_i)= \prod_{d\in \partial_i} \mathbb{S}\!\operatorname{to}_d (Q_i),
\]
where $\prod$ stands for the Cartesian product of the finitely many factors taken in some fixed order (say, for instance, positive once around $\partial_i$). 
See \cite[Section~7]{Boa4}. 

In~\cite[Section~8.1]{Boa4}, the notion of Stokes representation is introduced as a representation of the fundamental groupoid $\Pi$ of the irregular curve (having one base point near all the punctures and a further base point in the interior) in $G$ obeying some technical conditions. 
Their set is denoted by $\operatorname{Hom}_{\mathbb{S}}(\Pi , G)$. 
\begin{theorem}\label{thm:qH}
    $\operatorname{Hom}_{\mathbb{S}}(\Pi , G)$ is a smooth affine variety and carries a canonical quasi-Hamiltonian $H_1\times \cdots \times H_m$-space structure. 
\end{theorem}
We have an explicit description of $\operatorname{Hom}_{\mathbb{S}}(\Pi , G)$ as $\mu_G^{-1} (1)/G$, where $\mu_G$ is the $G$-valued moment map 
\begin{align*}
    \mu_G\colon \prod_{i=1}^m \left( G \times H_i \times \mathbb{S}\!\operatorname{to}(Q_i) \right) &\to G \\ 
    \{ (C_i, h_i , S^i_1, S^i_2, \ldots ) \}_{i=1}^m & \mapsto \mu_1 \cdots \mu_m, 
\end{align*}
with 
\[
 \mu_i = C_i^{-1} h_i \cdots S^i_2 S^i_1 C_i \in G. 
\]
\begin{theorem}[Riemann--Hilbert--Birkhoff correspondence]\cite[Corollary~A.4]{Boa4}\label{thm:RHB}
The isomorphism classes of meromorphic connections on algebraic principal $G$-bundles on $X$ with irregular type $Q_i$ at $a_i$ correspond bijectively to the $H_1\times \cdots \times H_m$-orbits in $\operatorname{Hom}_{\mathbb{S}}(\Pi , G)$. 
\end{theorem}
One may restrict attention to the set of isomorphism classes of connections whose associated Stokes representation is irreducible. 
By~\cite[Theorem~9.3]{Boa4}, in this case, irreducibility is equivalent to stability in the sense of Geometric Invariant Theory. 
It then follows from the theorem and Section~\ref{sec:GIT} that this subset carries a natural structure of regular holomorphic Poisson variety, called the wild character variety determined by $G$ and the irregular curve $(X, \mathbf{a}, \{ Q_i \}_{i=1}^m)$. 

We are particularly interested in the symplectic leaves of this wild character variety.
Any symplectic leaf is given as the quasi-Hamiltonian reduction of $\operatorname{Hom}_{\mathbb{S}}(\Pi , G)$ at $\prod_i \mathcal{C}_i$ for some set of conjugacy classes $\mathcal{C}_i \subset H_i$. 
(See the discussion after~\cite[Theorem~1.1]{Boa4}.)
Specifically, it is the set 
\begin{equation}\label{eq:symplectic_leaf}
    \mu_H^{-1}(\mathcal{C}_1 \times \cdots \times \mathcal{C}_m)/H_1\times \cdots \times H_m,
\end{equation}
where 
\begin{align*}
    \mu_H\colon \operatorname{Hom}_{\mathbb{S}}(\Pi , G) & \to H_1\times \cdots \times H_m \\ 
    \{ (C_i, h_i , S^i_1, S^i_2, \ldots ) \}_{i=1}^m & \mapsto (h_1, \ldots , h_m).
\end{align*}
The wild character variety determined by $G$, the irregular curve $(X, \mathbf{a}, \{ Q_i \}_{i=1}^m)$ and the conjugacy classes $\{ \mathcal{C}_i \}_{i=1}^m$ is the holomorphic symplectic variety~\eqref{eq:symplectic_leaf}.

\subsubsection{Stokes local systems -- twisted case}

Here, we lift the assumption $d=1$ in~\eqref{eq:HLT}, that is, the matrices $A_l$ occurring in~\eqref{eq:irreg} are no longer diagonalizable without passing to a ramified covering. 
We do not give full details because the results go completely parallel to the untwisted case, see~\cite{BY}. 
Instead, we just focus on the differences that occur. 

First, the singular directions (whose definition agrees with the one in~\eqref{eq:singular_direction}) at a twisted irregular singular point in the $t$-plane no longer necessarily come in opposite pairs, in particular there may be an odd total number of them. 
Namely, they do come in pairs in the $u$-plane (see~\eqref{eq:extension}), but as $u$ ranges over a full angle $2\pi$, $t$ only ranges over a sector of opening $2\pi/d$, where only some of the singular directions occur. 

Second, the role of the subgroups $H_i$ is played by the group of graded automorphisms of the trivial $G$-torsor on $\partial_i$, the grading coming from the irregular type.  
Assuming that the exponential factors are generic, this means that compared to a maximal torus, the dimension of $H_i\subset G$ drops by some amount. 
In practical terms, 
\begin{equation}\label{eq:twisted_torus}
    H_i = \prod_Q T_Q
\end{equation} 
where for each $Q$, $T_Q$ is a maximal torus of the regular singular constituents $\mathcal{F}_Q$ appearing in~\eqref{eq:HLT}. 
Notice that by maximal tori we mean those of the simple groups $\operatorname{SL}(r, \mathbb{C})$, because central elements act trivially. 

\subsubsection{Affine GIT and symplectic quotients}\label{sec:GIT}

Let us be given an affine scheme $S = \operatorname{Spec} R$ for a ring $R$.
Let $G$ be a reductive algebraic group and $K\subset G$ a maximal compact subgroup.  
Assume that $G$ acts on $S$ algebraically. 
Then the affine Geometric Invariant Theory quotient of $S$ by $G$ is defined to be the affine scheme $S/\!\!/G = \operatorname{Spec} R^G$, where 
\[
R^G = \{ r\in R \vert \, g\cdot r = r \; \mbox{for all} \; g\in G \}.
\]
On the other hand, when $S$ carries a natural symplectic structure and the restricted action of $K$ on $S$ is Hamiltonian, then there exists an equivariant momentum map 
\[
 \mu\colon S\to \mathfrak{k}^{\vee}. 
\]
The symplectic quotient of $S$ by $K$ is then 
\[
    \mu^{-1}(\zeta ) / K 
\]
for a central element $\zeta \in \mathfrak{k}^{\vee}$. 
We will frequently use the following result. 
\begin{theorem}[Kempf--Ness]~\cite{MFK}\label{thm:KN}
 The Geometric Invariant Theory quotient $S/\!\!/G$ is canonically isomorphic to the symplectic quotient $\mu^{-1}(\zeta ) / K$. 
\end{theorem}
We will apply the algebraic side of this theorem in Section~\ref{sec:Betti} to determine the symplectic leaves~\eqref{eq:symplectic_leaf}.

\section{Joshi-Kitaev-Treharne systems}\label{sec:JKT}

Consider $E\rightarrow X$ holomorphic vector bundle of rank 3, and $D=\sum r_ic_i$ effective divisor on $X$. Furthermore, consider a meromorphic connection $\nabla$ on $X$ such that it has given local form 
\begin{equation}
\label{eq:polar}
    \nabla = \operatorname{d} + \operatorname{d}\! Q_i + \Lambda_i \frac{\operatorname{d}\!z}{z}+O(1)\operatorname{d}\!z
\end{equation}
in some holomorphic local trivialization of $E$ around each $c_i$, where $Q_i$ is an irregular type as in~\eqref{eq:irreg}, $\Lambda_i$ is some constant $3\times 3$ matrix (called the residue matrix), and $O(1)$ refers to the holomorphic part. By general theory, the meromorphic connection $(E,\nabla)$ can be associated with the $\mathcal{D}$-module $\mathbb{M}$ over $\mathbb{C}(z)$, introduced in Section \ref{sec:connection}. Identifying $\Omega_{\mathbb{C}(z)/\mathbb{C}}$ with $\mathbb{C}(z)\operatorname{d}\! z$, the map $\nabla:E\rightarrow\Omega_{\mathbb{C}(z)/\mathbb{C}}\otimes E$ provides a differential module structure on $\mathbb{M}$, see~\cite[Observation~1.8]{vPS}.

Assuming that $\mathbb{M}$ corresponds to the meromorphic connection $\nabla$ in the above sense, $C=\textrm{Sing}(\mathbb{M})$ is the reduced part of the divisor $D$. Our aim is to apply this general theory to the special cases when we have either $D=\{0\}+2\cdot \{\infty\}$ or $D=3\cdot \{\infty\}$. In the first case, we say that there is a logarithmic (regular) singularity at $c=0$. Regarding the irregular singularity at $c=\infty$, we distinguish three possibilities in both cases, depending on whether the leading order term of the local form at $c=\infty$ ($A_2$ or $A_3$ respectively) is 
\begin{itemize}
    \item semisimple (we call this the untwisted case),
    \item consists of a Jordan block of size $2$ and a Jordan block of size $1$ (minimally twisted case),
    \item or consists of a Jordan block of size $3$ (maximally twisted case).
\end{itemize}

This provides us six different cases (summarized in Table~\ref{tab1}) of rank 3 connections, used by Joshi, Kitaev, and Treharne in~\cite{JKTI},~\cite{JKTII}, in order to give the $3\times 3$ Lax representations of the Painlev\'e equations. We will refer to these as $JKT*$-cases, where the symbol $*$ denotes one of the following:
\begin{equation*}
     \{VI,V,IVa,IVb,II,I\}.
\end{equation*}
Specifically, the local forms at $\infty$ of the meromorphic connection are fixed as follows, with respect to some holomorphic trivialization of $E$. 
In the cases $JKTVI, JKTV, JKTIVa$ corresponding to $D=\{0\}+2\cdot \{\infty\}$, the eigenvalues of $\operatorname{Res}_0 \nabla$ are fixed generic values, denoted by $\nu_i$, $0 \leq i \leq 2$.

\paragraph{JKTVI}
We have $D = 2 \cdot \{ \infty \} + \{ 0 \}$, such that the irregular singularity $c=\infty$ is untwisted, with local form at $c=\infty$ given by 
\begin{equation}
\label{polar1}\tag{JKTVI}
    \nabla=\operatorname{d}+\left[
    \begin{pmatrix}
        a_0 & 0 & 0 \\
        0 & a_1 & 0 \\
        0 & 0 & a_2
    \end{pmatrix}w^{-2}+\begin{pmatrix}
        b_0 & 0 & 0 \\
        0 & b_1 & 0 \\
        0 & 0 & b_2
    \end{pmatrix}w^{-1}+O(1)
    \right]\otimes\textrm{d}w,
\end{equation}
where $a_i,b_i\in\mathbb{C}$ ($i=0,1,2$) are fixed, $a_i$'s are mutually different, and 
\[
b_0+b_1+b_2=\operatorname{Tr Res}_{0} \nabla
\]
because of the residue theorem.
This local form is obviously equivalent to fixing the exponential factors: 
\begin{equation}\label{eq:qi_JKTVI}
    q_i = q_{i,1}w^{-1}=- a_i w^{-1}, \quad 0 \leq i \leq 2, 
\end{equation}
with corresponding residue eigenvalues $b_i$.

\paragraph{JKTV}
We have $D = 2 \cdot \{ \infty \} + \{ 0 \}$, such that the irregular singularity $c=\infty$ is minimally twisted, with local form at $c=\infty$ given by 
\begin{equation}
\label{polar2}\tag{JKTV}
    \nabla=\operatorname{d}+\left[
    \begin{pmatrix}
        a_0 & 1 & 0 \\
        0 & a_0 & 0 \\
        0 & 0 & a_1
    \end{pmatrix}w^{-2}+\begin{pmatrix}
        0 & 0 & 0 \\
        b_0 & b_1 & 0 \\
        0 & 0 & b_2
    \end{pmatrix}w^{-1}+O(1)
    \right]\otimes\textrm{d}w,
\end{equation}
where $a_0,a_1,b_i\in\mathbb{C}$ ($i=0,1,2$) are fixed, $b_0 \neq 0$, $a_0,a_1$ are different, and 
\[
b_1+b_2=\operatorname{Tr Res}_{0} \nabla
\] 
because of the residue theorem.
According to~\cite[Theorem~1.3]{KSz}, this minimally twisted local form is equivalent to fixing the exponential factors: 
\begin{gather}
    q_{0,1}=q_{\pm}=\lambda_2 w^{-1/2}-a_0w^{-1}\nonumber \\
    q_2=-a_1w^{-1}. \label{eq:qi_JKTV}
\end{gather}
with corresponding residue eigenvalues $0,b_1,b_2$. Here $q_{0,1}$ is a 2-valued function, depending on the values of the square root. The coefficient $\lambda_2$ depends only on $b_0,b_1$.
See~\cite[Equations~(4.5a),~(4.5b)]{JKTI}.

\paragraph{JKTIVa}
We have $D = 2 \cdot \{ \infty \} + \{ 0 \}$, such that the irregular singularity $c=\infty$ is maximally twisted, with local form at $c=\infty$ given by 
\begin{equation}
\label{polar3}\tag{JKTIVa}
    \nabla=\operatorname{d}+\left[
    \begin{pmatrix}
        a_0 & 1 & 0 \\
        0 & a_0 & 1 \\
        0 & 0 & a_0
    \end{pmatrix}w^{-2}+\begin{pmatrix}
        0 & 0 & 0 \\
        0 & 0 & 0 \\
        b_0 & b_1 & b_2
    \end{pmatrix}w^{-1}+O(1)
    \right]\otimes\textrm{d}w,
\end{equation}
where $a_0,b_i\in\mathbb{C}$ ($i=0,1,2$) are fixed, and $b_2=\operatorname{Tr Res}_{0} \nabla$, because of the residue theorem. This maximally twisted local form is equivalent to fixing the exponential factors: 
\begin{gather}
    q_{0,1,2}= \lambda_1w^{-1/3}+\lambda_2w^{-2/3}-a_0w^{-1} \label{eq:qi_JKTIVa} 
\end{gather}
with the corresponding residue eigenvalues $0,0,b_2$. Here, $q_{0,1,2}$ is a 3-valued function, depending on the values of the cubic root. The coefficients $\lambda_2,\lambda_1$ depend only on $b_0,b_1,b_2$.

\paragraph{JKTIVb}
We have $D = 3 \cdot \{ \infty \}$, such that the irregular singularity $c=\infty$ is untwisted, with local form at $c=\infty$ given by 
\begin{equation}
\label{polar4}\tag{JKTIVb}
    \nabla=\operatorname{d}+\left[
    \begin{pmatrix}
        a_0 & 0 & 0 \\
        0 & a_1 & 0 \\
        0 & 0 & a_2
    \end{pmatrix}w^{-3}+\begin{pmatrix}
        b_0 & 0 & 0 \\
        0 & b_1 & 0 \\
        0 & 0 & b_2
    \end{pmatrix}w^{-2}+\begin{pmatrix}
        c_0 & 0 & 0 \\
        0 & c_1 & 0 \\
        0 & 0 & c_2
    \end{pmatrix}w^{-1}+O(1)
    \right]\otimes\textrm{d}w,
\end{equation}
where $a_i,b_i,c_i\in\mathbb{C}$ ($i=0,1,2$) are fixed, $a_i$'s are mutually different, and $c_0+c_1+c_2=0$, because of the residue theorem.
This local form is equivalent to fixing the exponential factors: 
\begin{gather}\label{eq:qi_JKTIVb}
    q_i = -b_iw^{-1}-\frac{1}{2}a_iw^{-2}, \quad 0 \leq i \leq 2, 
\end{gather}
with corresponding residue eigenvalues $c_i$.

\paragraph{JKTII}
We have $D = 3 \cdot \{ \infty \}$, such that the irregular singularity $c=\infty$ is minimally twisted, with local form at $c=\infty$ given by  
\begin{equation}
\label{polar5}\tag{JKTII}
    \nabla=\operatorname{d}+\left[
    \begin{pmatrix}
        a_0 & 1 & 0 \\
        0 & a_0 & 0 \\
        0 & 0 & a_1
    \end{pmatrix}w^{-3}+\begin{pmatrix}
        0 & 0 & 0 \\
        b_0 & b_1 & 0 \\
        0 & 0 & b_2
    \end{pmatrix}w^{-2}+\begin{pmatrix}
        0 & 0 & 0 \\
        c_0 & c_1 & 0 \\
        0 & 0 & c_2
    \end{pmatrix}w^{-1}+O(1)
    \right]\otimes\textrm{d}w,
\end{equation}
where $a_0,a_1,b_i\in\mathbb{C}$ ($i=0,1,2$) are fixed, $a_0,a_1$ are different, and $c_1+c_2=0$, because of the residue theorem.
This minimally twisted local form is equivalent to fixing the exponential factors: 
\begin{gather}
     q_{0,1}=q_{\pm}=\lambda_1w^{-1/2}+\lambda_2w^{-1}+\lambda_3w^{-3/2}-\frac{1}{2}a_0w^{-2} \nonumber \\
     q_2=-b_2w^{-1}-\frac{1}{2}a_1w^{-2}. \label{eq:qi_JKTII}
\end{gather}
with corresponding residue eigenvalues $0,c_1,c_2$. Here, $q_{0,1}$ is a 2-valued function, depending on the values of the square root. The coefficients $\lambda_1,\lambda_2,\lambda_3$ depend only on $b_0,b_1,c_0,c_1$.

\paragraph{JKTI}
We have $D = 3 \cdot \{ \infty \}$, such that the irregular singularity $c=\infty$ is maximally twisted, with local form at $c=\infty$ given by 
\begin{equation}
\label{polar6}\tag{JKTI}
    \nabla=\operatorname{d}+\left[
    \begin{pmatrix}
        a_0 & 1 & 0 \\
        0 & a_0 & 1 \\
        0 & 0 & a_0
    \end{pmatrix}w^{-3}+\begin{pmatrix}
        0 & 0 & 0 \\
        0 & 0 & 0 \\
        b_0 & b_1 & b_2
    \end{pmatrix}w^{-2}+\begin{pmatrix}
        0 & 0 & 0 \\
        0 & 0 & 0 \\
        c_0 & c_1 & c_2
    \end{pmatrix}w^{-1}+O(1)
    \right]\otimes\textrm{d}w,
\end{equation}
where $a_0,b_i,c_i\in\mathbb{C}$ ($i=0,1,2$) are fixed, and $c_2=0$, because of the residue theorem.
This maximally twisted local form is equivalent to fixing the exponential factors: 
\begin{gather}\label{eq:qi_JKTI}
    q_{0,1,2} =\lambda_1w^{-1/3}+\lambda_2w^{-2/3}+\lambda_3w^{-1}+\lambda_4w^{-4/3}+\lambda_5w^{-5/3}-\frac{1}{2}a_0w^{-2},
\end{gather}
with corresponding residue eigenvalues $0,0,c_2$. Here, $q_{0,1,2}$ is a 3-valued function, depending on the values of the cubic root. The coefficients $\lambda_i$ ($1\leq i\leq 5$) depend only on $b_i,c_i$. 

The equivalence between the above local normal forms and the exponential factors is formulated in the following lemma.

\begin{lemma}
    Assume that the exponential factors are of the forms \eqref{eq:qi_JKTVI}-\eqref{eq:qi_JKTI}. 
    Then there exist polynomial gauge transformations in the indeterminate $w$, such that the local forms of the connection matrix~\eqref{eq:polar} in the six cases are in the normal forms \eqref{polar1}, \eqref{polar2}, \eqref{polar3}, \eqref{polar4}, \eqref{polar5}, \eqref{polar6}, respectively. 
    Conversely, if the local form is given by \eqref{polar1}, \eqref{polar2}, \eqref{polar3}, \eqref{polar4}, \eqref{polar5} or \eqref{polar6}, then the exponential factors are given by~\eqref{eq:qi_JKTVI}-\eqref{eq:qi_JKTI}. 
\end{lemma}
\begin{proof}
    The particular cases of interest to us can be calculated directly using the quadratic and cubic formulas. 
    See~\cite[Theorem 2.1]{KSz} for a more general statement in arbitrary rank.
    Notice here that even though~\cite{KSz} deals with Higgs bundles rather than connections, for the gauge transformations $g$ considered in \emph{loc. cit.}, the extra terms $g^{-1} \operatorname{d}\! g$ that appear in gauge transformations of connections do not affect the polar part (i.e., the exponential factors and the eigenvalues of the residue) of the connection matrix. 
    Indeed, in~\cite{KSz} the only gauge transformations $g$ that are used are finite products of factors of the form $\operatorname{I} + g_l z^l$ with $l>0$, for which 
    \begin{align*}
    g^{-1} \operatorname{d}\! g & = \operatorname{d} \ln g \\
    & = \sum_l \operatorname{d} \ln (\operatorname{I} + g_l z^l) \\
    & = \sum_l \operatorname{d} (- g_l z^l + g_l^2 z^{2l}/2 - \cdots ) \\
    & = \sum_l (- l g_l z^{l-1} + \cdots )
    \end{align*}
    has only positive terms in its Laurent expansion.
\end{proof}

\begin{defn}
    Let $\mathcal{M}_{dR}^{JKT*}$ be the (de Rham) moduli space of meromorphic integrable connections with fixed divisor $D$ on $\mathbb{C}P^1$ and with fixed local form at $D$ (including fixed generic values of parameters $a_i,b_i,c_i$) of the just described $JKT*$ system.
\end{defn}

Because of \cite[Theorem 0.2.]{BB} (or rather, a suitable extension~\cite{ESz_Fourier} to the twisted case), these moduli spaces exist and they are hyperK\"ahler manifolds. 
Moreover, the cases $\operatorname{rank} (E) = 3$ and $\operatorname{length} D = 3$ are related to the root system $\widetilde{E}_6$. The cases where the moduli spaces are of the lowest possible complex dimension $2$ have been listed in~\cite[Section~4.1]{Boa3}. 
The six $JKT*$ cases are the degenerations (confluences of singularities) of the rank 3 case with 3 logarithmic singularities, which we analyzed in \cite{ESz}. 

\begin{center}
\begin{tabular}{ |c|c|p{4cm}|  }
 \hline
 \multicolumn{3}{|c|}{Summary of the investigated cases} \\
 \hline
 $JKT$ system& polar divisor of $\nabla$ &type of irregular singularity at $c=\infty$\\
 \hline
 \hline
 $JKTVI$   & $D=\{0\}+2\{\infty\}$    &untwisted\\
 \hline
 $JKTV$ &   $D=\{0\}+2\{\infty\}$  & minimally twisted   \\
 \hline
 $JKTIVa$ &$D=\{0\}+2\{\infty\}$ & maximally twisted\\
 \hline
 $JKTIVb$    &$D=3\{\infty\}$ & untwisted\\
 \hline
 $JKTII$&   $D=3\{\infty\}$  & minimally twisted\\
 \hline
 $JKTI$& $D=3\{\infty\}$  & maximally twisted  \\
 \hline
\end{tabular}
\captionof{table}{The six irregular cases investigated}
\label{tab1}
\end{center}

\section{Betti spaces}\label{sec:Betti}

When $c\in \CP1$ is a regular singularity of the differential module $\mathbb{M}$, the regular version of the Riemann--Hilbert correspondence associates with $\mathbb{M}$ the local system of solutions in a punctured disc $\Delta^*$ around $c$. 
This is equivalent to simply the monodromy data of $\mathbb{M}_c$, which is the representation of the fundamental group $\mathbb{Z}$ of $\Delta^*$ in $G$ given by analytic continuation of a given solution. 
It is determined up to conjugation. 
However, in case of an irregular singular module at $c$, the monodromy transformation is not enough to describe the module, because the asymptotic behavior of the solutions change at the singular directions, so Stokes data have to be considered, too. 

See Subsection~\ref{wild character varieties} for the description of the Stokes local system associated with an irregular connection $(E,\nabla)$. 
Let
\[
q_i=\lambda_{1}w^{-1/N }+\cdots +\lambda_{k}w^{-k/N }, \quad q_j=\lambda_{1}'w^{-1/N}+\cdots +\lambda_{k}'w^{-k/N }
\]
be the exponential factors of $\nabla$ at $c=\infty$, and consider their difference 
\[
q_i-q_j=(\lambda_{1}-\lambda_{1}')w^{-1/N}+...+(\lambda_{l}-\lambda_{l}')w^{-l/N}
\]
for some $k,l,N \in \mathbb{Z}_+$, so that $(\lambda_{l}-\lambda_{l}')\neq 0$, and at least one of $\lambda_{k},\lambda_{k}'$ is nonzero. In the untwisted case $N=1$, it follows from our assumptions (for instance, that the $a_i$'s are different in~\eqref{polar1}) that $k=l$. 
Let us use the notation $\varepsilon=e^{2\pi\sqrt{-1}/N}$ for a primitive $N$-th root of unity and let $w^{1/N}$ be a fixed sheet of the $N$-valued function on some angular sector of opening $<2\pi$. 
In the twisted cases $N>1$, and if $q_i,q_j$ belong to the same group of roots under ramification, then for all $1\leq s \leq k$ we have $\lambda_s' = \lambda_s \varepsilon^{-as}$ for some $a\in \mathbb{Z}$. 
In this case, $k>l$ holds if $N\vert k$. 
Recall now from~\cite[Definition 7.18]{vPSg} that the Stokes directions for the pair $\{q_i-q_j\}$ form the set of $\varphi\in S^1$ that are tangent at $0$ to the smooth components of the curve 
\begin{equation}
\label{eq:realpart}
    \textrm{Re}\left((\lambda_{l}-\lambda_{l}')w^{-l/N}\right)=0, \hspace{0.5cm} w=\vert w \vert e^{\sqrt{-1}\varphi}.
\end{equation}
In general, the Stokes directions are the solutions $\varphi$ of 
\begin{equation}
\label{Stokesdir}
\operatorname{Arg} (\lambda_{l}-\lambda_{l}') - \frac{l \varphi}N \in \frac{(2\mathbb{Z}+1)\pi}{2}.
\end{equation}
Meanwhile, singular directions \eqref{eq:singular_direction} are the angle bisectors of the sectors between consecutive Stokes directions, given by solutions of 
\begin{equation}\label{eq:sing}
\operatorname{Arg} (\lambda_{l}-\lambda_{l}') - \frac{l \varphi}N \in \mathbb{Z}\pi. 
\end{equation}

\begin{rmrk}
\label{rmrk:Stokes}
    We have 
    \[
     \operatorname{Arg} (\lambda_{l}-\lambda_{l}') = \operatorname{Arg} (\lambda_{l}) + \operatorname{Arg}(1-\varepsilon^{-al}).
    \]
    The term $\operatorname{Arg} (\lambda_{l})$ in this expression just shifts all phases, i.e. rotates the Stokes and singular directions, but otherwise does not change their arrangement. 
    That is, we can make a specific choice for coefficients $\lambda_{l}$ without loss of generality to compute the Stokes and singular directions. 
\end{rmrk}
We refer to~\cite[Section~7]{vPSg} for the notion of Stokes map $S_{\varphi}\in\textrm{SL}(n,\mathbb{C})$ associated to a connection with irregular singularity at $c$ and singular direction $\varphi$. 
Loosely speaking, $S_{\varphi}$ is the unipotent matrix describing the change of preferred solutions (i.e., ones that have the expected asymptotic expansions) from one side of the singular direction to the other. 
Moreover, let $L_c\in\textrm{GL}(n,\mathbb{C})$ be the parallel transport map from a fixed base point $c_0$ to $c$ (called the link automorphism).  
The monodromy automorphism (or topological monodromy) about $c$ is then given by 
\begin{equation}
\label{holonomy}
    M_c=L_cH\left(\prod_{\varphi} S_{\varphi}\right) L_c^{-1},
\end{equation}
where $\varphi$ ranges through the set of singular directions in positive orientation,  and $H$ is the formal monodromy. 
We will explicitly spell out the formula for the formal monodromy below, case by case. 

\begin{defn}
The \emph{exponential torus} of an irregular type~\eqref{eq:irreg} is the centralizer of $A_{r_i}$ in $\operatorname{PGL}(3, \mathbb{C})$. 
\end{defn}
For the exponential tori in the $JKT*$ cases, see Lemma~\ref{lem:exptorus}.
Recall~\cite[Theorem~8.10]{vPSg} that the differential Galois group of the differential module corresponding to $\nabla$ over $\mathbb{C}(z)$ is the linear algebraic group generated by the formal differential Galois group and the Stokes matrices. The formal differential Galois group is the differential Galois group over $\mathbb{C}(\!(z)\!)$, which is generated, as a linear algebraic group, by the exponential torus and the formal monodromy. (Since in the cases we consider there is always only a single irregular singularity, it is not necessary to separate into local cases around different singularities.)

To define the formal monodromy of $\nabla$, consider the universal differential ring extension
\begin{equation*}
    \textrm{Univ}R:=\mathbb{C}(\!(z)\!)\left[z^a,\textrm{log}z\right]\textrm{, where $a\in\mathbb{C}$, $z^1=z$, $z^az^b=z^{a+b}$}
\end{equation*}
with the $\mathbb{C}(\!(z)\!)$-linear differential automorphism $H\in\textrm{GL}(n,\mathbb{C})$, defined by the action $Hz^a=e^{2\pi\sqrt{-1}a}z^a$, and $H(\textrm{log}z)=\textrm{log}z+2\pi\sqrt{-1}$, called the formal monodromy. 
Now consider the free $\textrm{Univ}R$-module 
\begin{equation*}
    \overline{\textrm{Univ}R}:=\bigoplus_{i=1}^n \textrm{Univ}R\cdot e(q_i),
\end{equation*}
with the extension $He(q_i)=e(Hq_i)$. 
Then, the solution space of the matrix differential equation $\nabla=\operatorname{d} + A$, where $A\in\mathfrak{g}(n,\mathbb{C}(\!(z)\!))$, is
\begin{equation*}
    V=\{v\in\overline{\textrm{Univ}R}\vert(\operatorname{d}+A)v=0\}=\bigoplus_{i=1}^n V_{q_i}
\end{equation*}
corresponding to the above decomposition, and obviously $H(V_{q_i})=V_{Hq_i}$ for all $i$. 
In order to obtain the concrete formal monodromies that we are interested in, let us distinguish three cases.
\begin{itemize}
    \item In the untwisted case the subspaces $V_{q_i}$ are invariant under the formal monodromy, so its matrix is diagonal, namely
    \begin{equation*}
        H_1=\textrm{exp}\left(2\pi\sqrt{-1}\textrm{Diag}\left[\textrm{Res}_{c}\nabla\right]\right)=
        \begin{pmatrix}
\alpha & 0 & 0\\
0 &  \beta & 0\\
0 & 0 & \gamma
\end{pmatrix}.
    \end{equation*}
where $\textrm{Diag}\left[\textrm{Res}_{c}\nabla\right]$ means the diagonal part of the residue of $\nabla$ at $c$, which is fixed by our assumption. Hence $\alpha\beta\gamma$ is also fixed, say $\alpha\beta\gamma=1$ (without loss of generality). 
    \item In the minimally twisted case $V=V_{q_+}\oplus V_{q_{-}}\oplus V_{q_2}$, where $H_2$ exchanges the subspaces $V_{q_+}$ and $V_{q_-}$, while $V_{q_2}$ is invariant. There is a basis $\{e_+,e_-,e_2\}$ of $V$, such that $H_2e_2=\alpha e_2$ for some $\alpha\in\mathbb{C}^{\times}$ and $e_+,e_-$ are chosen in the way that $H_2e_+=e_-$. By the assumption $H_2\in\textrm{SL}(3,\mathbb{C})$, we have necessarily $H_2e_-= -\alpha^{-1}e_+$, and 
    \[
H_2= \begin{pmatrix}
0 & -\alpha^{-1} & 0\\
1 &  0 & 0\\
0 & 0 & \alpha
\end{pmatrix} .
\]
    
    \item In the maximally twisted case the $V_{q_i}$ subspaces are cyclically permuted by $H_3$, and we can choose a basis $\{e_0,e_1,e_2\}$ in $V$, such that $H_3e_0=e_1$, and $H_3e_1=e_2$. Again by the determinant condition we have $H_3e_3=e_1$, and 
    \[
H_3= \begin{pmatrix}
0 & 0 & 1\\
1 &  0 & 0\\
0 & 1 & 0
\end{pmatrix} .
\]
\end{itemize}

As we said, the differential Galois group for $\nabla$ is generated by the exponential torus and the formal monodromy in addition to the Stokes matrices. That is, we have to consider the action of the maximal torus of the $G$ structure group on the Stokes data, which is a $(\mathbb{C}^{\times})^k$-action for $k\geq0$ integer. This action is described by the following lemma. 

\begin{lemma}\label{lem:exptorus}
    \begin{enumerate}
        \item[i)] The dimension of the exponential torus is $k=2,1,0$ in the untwisted, the minimally twisted, and the maximally twisted cases respectively.
        \item[ii)] When $D = 2\cdot \{ \infty \}$, that is, in the cases $VI, V, IVa$, the $G$-invariant elements of the coordinate ring of the set of Stokes data are generated by: $x_1x_4,x_2x_5,x_3x_6,x_1x_3x_5,x_2x_4x_6$. Here the notation for the elements of the Stokes matrices, corresponding to the pairs $\{q_i-q_j\}$, is encoded by:
        \begin{center}
\begin{tabular}{ c|c|c|c|c|c|c } 

 pairs & $\{q_0-q_1\}$ & $\{q_0-q_2\}$ & $\{q_1-q_2\}$ & $\{q_1-q_0\}$ & $\{q_2-q_0\}$ & $\{q_2-q_1\}$ \\ 
 \hline
 elements & $x_1$ & $x_2$ & $x_3$ & $x_4$ & $x_5$ & $x_6$ \\ 
\end{tabular}
\end{center} 
\item[iii)]
When $D = 3\cdot \{ \infty \}$, i.e. in the cases $IVb, II, I$, further coefficients $x_7,...,x_{12}$ need to be introduced too. 
The invariant polynomials are generated by the same monomials as in the previous part and the same table applies, up to adding $6$ to any of the indices of the variables $x_i$. 
    \end{enumerate}
\end{lemma}

\begin{proof}
    \begin{itemize}
        \item[i)] 
        In the untwisted case, we have two degrees of freedom in choosing the parameters $\alpha,\beta,\gamma$, while the subspaces $V_{q_i}$, $i=0,1,2$ are unique. The possible base changes are $\{e_0,e_1,e_2\}\leftrightarrow \{\lambda e_0,\mu e_1,(\lambda\mu)^{-1}e_2\}$, for $(\lambda,\mu)\in(\mathbb{C}^{\times})^2$, and the group action on the data $\{x_i\}$ is two-dimensional.

        In the minimally twisted case, the possible base changes are given by $\{e_+,e_-,e_2\}\leftrightarrow \{\lambda e_+,\lambda e_-,\lambda^{-2}e_2\}$ for $\lambda\in\mathbb{C}^{\times}$, because of the described action of $H$. Therefore, the action on $\{x_i\}$, decreases the degree of freedom by one.

        In the maximally twisted case, there is no degree of freedom in the action of $H$, the chosen basis is unique, and therefore the group action on $\{x_i\}$ is trivial.
        \item[ii)] Consider the $G$-action on the Stokes matrices, where
        \begin{equation*}
            G=\left\{\begin{pmatrix}
                \lambda & 0 & 0\\
                0 &  \mu & 0\\
                0 & 0 & 1
            \end{pmatrix}\colon\quad \lambda,\mu\in\mathbb{C}^{\times}\right\}
        \end{equation*}
        (assuming $\mu=1$ in the minimally twisted case and $\lambda=\mu=1$ in the maximally twisted case). That is
        \begin{equation*}
            \begin{pmatrix}
                \lambda & 0 & 0\\
                0 &  \mu & 0\\
                0 & 0 & 1
            \end{pmatrix}\cdot \begin{pmatrix}
                \alpha & x_1 & x_2\\
                x_4 &  \beta & x_3\\
                x_5 & x_6 & \gamma
            \end{pmatrix}\cdot \begin{pmatrix}
                \lambda^{-1} & 0 & 0\\
                0 &  \mu^{-1} & 0\\
                0 & 0 & 1
            \end{pmatrix} = \begin{pmatrix}
                \alpha & \lambda\mu^{-1}x_1 & \lambda x_2\\
                \lambda^{-1}\mu x_4 &  \beta & \mu x_3\\
                \lambda^{-1}x_5 & \mu^{-1}x_6 & \gamma
            \end{pmatrix}
        \end{equation*}
        (here the indices of the $x_i$'s are to be understood modulo 6). This shows that we have six roots in the corresponding root system $A_2$. We can associate weight vectors to the variables $x_i$ for the $(\mathbb{C}^{\times})^2$-action. This is represented in Figure \ref{fig:fig7}.
        \begin{figure}[ht]
\centering
\includegraphics[width=8.0cm]{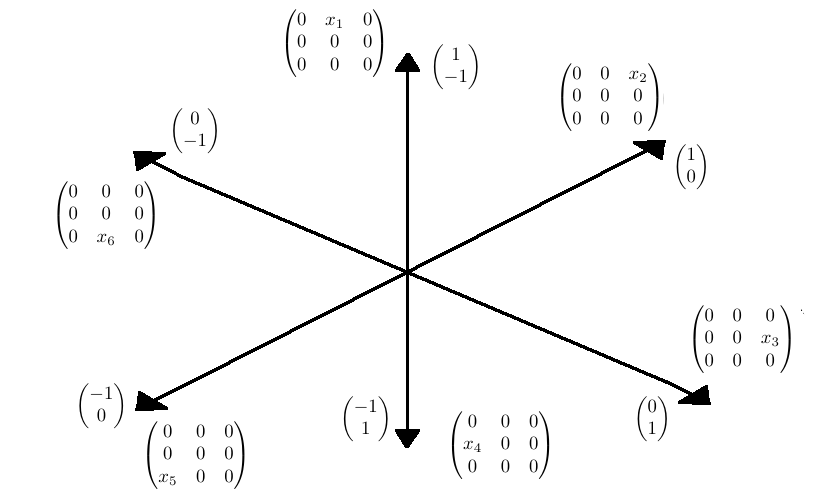}
\caption{Root system with the weight vectors}
\label{fig:fig7}
\end{figure}
The ring of invariant elements of the coordinate ring of Stokes data is generated by the monomials where the sum of the weight vectors add up to 0. These are: $x_1x_4$, $x_2x_5$, $x_3x_6$, $x_1x_3x_5$, $x_2x_4x_6$ (possibly some of the $x_i$'s are fixed to be $1$ in the twisted cases). Notice that these are precisely the monomials such that the corresponding pairs  $\{q_i-q_j\}$ add up to $0$ as well.
\item[iii)] Similar to ii). We leave it to the reader to check the full details. 
\end{itemize}
\end{proof}

\begin{rmrk}
\label{rmrk:variables}
    The second point ii) of the above lemma illustrates the statement of point i). In the untwisted case, the coordinate ring of the Stokes data under the group action reads as 
    \begin{equation*}
        \mathbb{C}[x_1,...,x_6]^G=\mathbb{C}[U,V,W,R,T]/\mathcal{I}
    \end{equation*}
      because of the non-trivial relation $UVW=RT$, where we have introduced the new variables for the invariant monomials: $U:=x_1x_4,V:=x_2x_5,W:=x_3x_6,R:=x_1x_3x_5,T:=x_2x_4x_6$ and the principal ideal is $\mathcal{I}=(UVW-RT)$. 
      We see that the dimension of the coordinate ring decreased by two under the action. 

      In the minimally twisted case (as we shall see) the Stokes coefficient $x_4$ does not appear, and we have $U:=x_2x_5,V:=x_3x_6,W:=x_1,R:=x_2x_6,T:=x_3x_5$ (or $T:=x_1x_3x_5$). 
      That is to say, the dimensions of both the original coordinate ring and the group $G$ fall by one, and we have:
      \begin{equation*}
        \mathbb{C}[x_1,x_2,x_3,x_5,x_6]^G=\mathbb{C}[U,V,W,R,T]/\mathcal{I}
    \end{equation*}
    where $\mathcal{I}=(UV-RT)$ (or $\mathcal{I}=(UVW-RT)$ depending on the choice of $T$).

    In the maximally twisted case (as we shall see) neither $x_5$ nor $x_6$ are present in the Stokes matrices, and the coordinate ring again has the same dimension $4$ after taking quotient:
    \begin{equation*}
        \mathbb{C}[x_1,x_2,x_3,x_4]^G=\mathbb{C}[U,V,W,R,T]/\mathcal{I}
    \end{equation*}
    where $U:=x_1x_4,V:=x_2,W:=x_3,R:=x_1x_3,T:=x_2x_4$, and the ideal is $\mathcal{I}=(UVW-RT)$. 
    We note that here the introduction of the new variables is unnecessary and will be omitted in the actual computations.
\end{rmrk}

Globally, we have a finite set $\{c_1,...,c_m\}$ of singular points, thus the above data, and the monodromy automorphisms $T_{c_i}$ can be determined at all of them. 
The global condition on them is that their product equals to the identity:
\begin{equation*}
    M_{c_m}\cdots M_{c_1}=I
\end{equation*}
In the cases we are interested in, the number of singular points is either $m=1$ or $m=2$. 
If $m=1$, the above condition means that the topological monodromy $M_{c_1}$ is equal to the identity:
\begin{equation}
    \label{identity}
    M_{c_1}=H\left(\prod_{\varphi} S_{\varphi}\right)=I.
\end{equation}
If $m=2$, one of the singularities (say $c_2$) is logarithmic, so the Stokes data are relevant only at $c_1$. 
The above condition then means that the topological monodromy $M_{c_1}$ of the irregular singularity belongs to a given conjugacy class. 
For a generic semi-simple conjugacy class in $\operatorname{SL}(3,\mathbb{C})$, this is in turn equivalent to its characteristic coefficients being fixed:
\begin{equation}
    \label{conjugacy}
    \textrm{Tr}(M_{c_1})=p, \hspace{0.5cm} \textrm{Tr}(M_{c_1}^2)=q
\end{equation}
(note that its determinant is necessarily $1$), for suitable $p,q\in\mathbb{C}$.

\begin{defn}
    The wild character variety or Betti moduli space $\mathcal{M}_B^{JKT*}$  of the respective $JKT*$ system is the affine GIT-quotient of the $(\mathbb{C}^{\times})^k$-action on the set of Stokes data, satisfying \eqref{identity} or \eqref{conjugacy}, depending on whether $m=1$ or $m=2$, defined in Section \ref{wild character varieties} (see, in particular, Theorem~\ref{thm:KN}).
\end{defn}

The Riemann--Hilbert--Birkhoff correspondence (Theorem~\ref{thm:RHB}) applies and establishes a $\mathbb{C}$-analytic isomorphism  
\begin{equation*}
    RH:\mathcal{M}_{dR}^{JKT*}\rightarrow\mathcal{M}_B^{JKT*}
\end{equation*}
between the de Rham and Betti spaces we investigate here.

In the sequel, we will determine exactly, in all the six investigated cases, the structure of the Stokes data and the equation of the wild character variety.
In each case, we find an affine cubic surface, or more precisely a family of affine cubic surfaces over some parameter space $\mathcal{P}$ (summarized in Table \ref{tab2}). 
As we will see, these equations turn out to be exactly the same as those of the corresponding irregular rank two local systems on $\mathbb{C}P^1$ belonging to JM representations of Painlev\'e isomonodromic families, described in \cite{vPS}. 
We analyze the parameters that appear in the equations of the surfaces and give the correspondence with the parameters in the rank two cases in \cite{vPS}.

\subsection{The logarithmic case}

For the sake of completeness, we mention this case, even though it does not belong to the cases we investigate here. 
In this case, the connection has three logarithmic singular points on $\mathbb{C}P^1$. 
In our previous paper, we completely described this case from the Betti point of view \cite[Section 3]{ESz}. 
Here, given that we only have regular singular points, the Stokes phenomenon does not play a role, and the monodromy data can be described by the traces of monodromy matrices on based loops around the three points (which we called Lawton's trace coordinates \cite{L}). 
Note that in arbitrary rank,~\cite[Theorem~1.3]{Pro} states that traces of products of monodromy matrices generate the invariant ring. 

It is worth to mention that here the Betti moduli space is also an affine cubic surface, but its projective compactification is slightly different from the other cases, namely its compactifying curve has 
\begin{equation*}
    0=-XYZ+X^3+Y^3
\end{equation*}
as equation.
In contrast to the irregular cases (where the compactifying curve is a chain of three rational curves), this has a unique component, a nodal rational curve.

\subsection{The $JKTVI$ case}

As we mentioned before, it was proven in~\cite{Har} that Painlev\'e VI also arises as isomonodromy deformation of a rank three system of linear differential equations with an irregular singularity of Poincar\'e rank $1$ and a logarithmic singularity.
This case was partially analyzed in the paper \cite[Example 2.2]{vPS}, 
see also \cite[Section~11.1]{Boa4}. 
Here, we describe the Stokes data and the wild Betti space.

We recall the setup from Section~\ref{sec:JKT}: $c_2=0$ is a logarithmic singularity with eigenvalues of its residue equal to $(\nu_1, \nu_2, \nu_3)$. In addition to this, at $c_1=\infty$, there is an untwisted irregular singularity of the form~\eqref{polar1}. 
The exponential factors at $c_2$ can be derived from formula \eqref{eq:qi_JKTVI}. 
Namely, these are $q_0=z = w^{-1},q_1=\varepsilon q_0 ,q_2=\varepsilon^2 q_0$, where $\varepsilon=e^{2\pi i/3}$ is the cubic root of unity. Here we set $(a_0,a_1,a_2)=(1,\varepsilon,\varepsilon^2)$ by virtue of Remark~\ref{rmrk:Stokes}. 
The singular directions are determined by the formula~\eqref{eq:sing} for all possible pairs $\{q_i-q_j\}$, in the order: $\{q_0-q_1\}$, $\{q_0-q_2\}$, $\{q_1-q_2\}$, $\{q_1-q_0\}$, $\{q_2-q_0\}$, $\{q_2-q_1\}$, and the directions are $\varphi=k\frac{\pi}{3}+\frac{\pi}{6}$, $k=1,...,6$, with corresponding Stokes matrices:
\begin{gather*}
    S_1= \begin{pmatrix}
1 & x_1 & 0\\
0 & 1 & 0\\
0 & 0 & 1
\end{pmatrix}, 
S_2= \begin{pmatrix}
1 & 0 & x_2\\
0 & 1 & 0\\
0 & 0 & 1
\end{pmatrix}, 
S_3= \begin{pmatrix}
1 & 0 & 0\\
0 & 1 & x_3\\
0 & 0 & 1
\end{pmatrix}, 
\\ S_4= \begin{pmatrix}
1 & 0 & 0\\
x_4 & 1 & 0\\
0 & 0 & 1
\end{pmatrix}, 
S_5= \begin{pmatrix}
1 & 0 & 0\\
0 & 1 & 0\\
x_5 & 0 & 1
\end{pmatrix}, 
S_6= \begin{pmatrix}
1 & 0 & 0\\
0 & 1 & 0\\
0 & x_6 & 1
\end{pmatrix}.
\end{gather*}
Thus, the topological monodromy here is $M_{\infty}=H_1\prod_{i=6}^1S_i$, and denoting the monodromy at $0$ by $M_0$, the condition $M_0 M_{\infty}=I$ holds. Recall that now we have a two-dimensional torus action, and the invariant monomials are $U:=x_1x_4$, $V:=x_2x_5$, $W:=x_3x_6$, $R:=x_1x_3x_5$, $T:=x_2x_4x_6$. 
The moduli space of these monodromy data is therefore parameterized by the above variables satisfying the tautological relation $UVW=RT$.
The system of equations~\eqref{conjugacy}, complemented by this tautological relation, reads as:
\begin{gather}
    \alpha + \beta U + \beta + \gamma W + \gamma T + \gamma + \gamma V = p \label{p61} \\
    \alpha^2 + 2\alpha\beta U + 2\alpha\gamma T + 2\alpha\gamma V + \beta^2 U^2 + 2\beta^2 U + \beta^2 + 2\beta\gamma W + 2\beta\gamma UT \nonumber \\
    + 2\beta\gamma T + 2\beta\gamma UW + 2\beta\gamma R + 2\beta\gamma UV + \gamma^2 W^2 + \gamma^2T^2 + 2\gamma^2 WT + 2\gamma^2W \nonumber \\
    + 2\gamma^2 T + \gamma^2 + \gamma^2 V^2 + 2\gamma^2 V + 2\gamma^2 VW + 2\gamma^2 VT = q \label{p62} \\
    UVW=RT \label{p63}
\end{gather}
We have three equations in five variables. 
Eliminating two variables (say $U$ and $R$ using \eqref{p61} and \eqref{p62}), we receive a cubic expression with leading order (i.e., degree $3$ homogeneous) term $\gamma VWT + \gamma V^2W + \gamma VW^2=\gamma VWS$, up to substituting $S=V+W+T$. 
This way, the degree $2$ homogeneous part stands as:
\begin{equation*}
    -\alpha W^2 -\beta V^2 - \gamma S^2 + (\alpha+\gamma)SW+(\beta+\gamma)SV+(-\alpha-\beta)VW.
\end{equation*}
Let us now use the linear rescaling $W=-X-(\beta+\gamma)$, $V=-Y-(\alpha+\gamma)$, and $S=-Z-(-\alpha-\beta)$. 
The equation of the surface is then 
\begin{equation*}
    {\mathcal{M}}_B^{JKTVI}=\{(X,Y,Z)\in\mathbb{C}^3|\gamma XYZ+\alpha X^2+\beta Y^2 + \gamma Z^2 + c_1X+c_2Y+c_3Z+c_4=0 \},
\end{equation*}
where $c_i=c_i(p,q,\alpha,\beta,\gamma)\in\mathbb{C}$ are some constants ($i=1,2,3,4$). 
\begin{remark}\label{rem:JKTVI}
    For special choice of parameters (say $(\alpha,\beta,\gamma)=(\varepsilon^2,\varepsilon,1)$), we recover the equation of the Painlev\'e VI-type JM Betti moduli space of~\cite[Subsection 3.1]{vPS}.
    More precisely, we get a family of affine cubic surfaces ${\mathcal{M}}_B^{JKTVI}\rightarrow\mathbb{C}^4$ over the parameter space $\mathcal{P}=\mathbb{C}^4$.
\end{remark}

\subsection{The $JKTV$ case}\label{sec:Betti_VI}

In this case, similarly to the previous one, there is a logarithmic singularity at $c_2=0$, with the same condition on the eigenvalues, and an irregular one at $c_1=\infty$, where the ramification index is 1 (minimally twisted), see \eqref{polar2}. That is, the local form is not diagonal, but block-diagonal, and the exponential factors are given in \eqref{eq:qi_JKTV}. The $1\times 1$ block has the exponential factor $q_2=-a_1w^{-1}$, and the $2\times 2$ block has the (double-valued) exponential factor $q_{0,1}=-a_0w^{-1}+\lambda_2w^{-1/2}$. Recall from Remark \ref{rmrk:Stokes} that for ease of computation we may assume $\textrm{Re}(a_0)=\textrm{Re}(a_1)$. 
Then the singular direction for $\{q_{0}-q_{1}\}$ is the solution of $\textrm{Im}(w^{-1/2})=0$, which is $\varphi=0$. On the other hand, for the differences $\{q_2-q_{0,1}\}$ and $\{q_{0,1}-q_2\}$ \eqref{eq:sing} is equivalent to $\varphi=\frac{\pi}{2}$ and $\varphi=\frac{3\pi}{2}$. This also means that the real part of $q_2$ becomes equal simultaneously to both value of $q_{0,1}$, and the corresponding Stokes matrices have two off-diagonal elements. Indeed, these are in appropriate order:
\begin{gather*}
    S_1= \begin{pmatrix}
1 & 0 & x_2\\
0 & 1 & x_3\\
0 & 0 & 1
\end{pmatrix}, 
S_2= \begin{pmatrix}
1 & x_1 & 0\\
0 & 1 & 0\\
0 & 0 & 1
\end{pmatrix}, 
S_3= \begin{pmatrix}
1 & 0 & 0\\
0 & 1 & 0\\
x_5 & x_6 & 1
\end{pmatrix}.
\end{gather*}
For the topological monodromies $M_0$ and $M_{\infty}=H_2S_3S_2S_1$ the same conditions apply as in Section~\ref{sec:Betti_VI}. The moduli space of these monodromy data is parameterized by the variables $x_i,(i=1,2,3,5,6)$ ($x_4$ is not present here, see Remark \ref{rmrk:variables}).
An equivariant set of parameters is $(U,V,W,R,T)$ satisfying the conditions $UV=RT$, and \eqref{conjugacy}; here $U:=x_2x_5$, $V:=x_3x_6$, $W:=x_1$, $R:=x_2x_6$, $T:=x_3x_5$ . 
Thus the system of equations determining the complex surface is
\begin{gather}
    \alpha + \alpha U +  W + \alpha V + \alpha TW = p \label{p51} \\
    -2\alpha^{-1} - 2T + 2\alpha R + 2\alpha VW + W^2 + 2\alpha WU + 2\alpha W^2T + \alpha^2V^2 \nonumber \\ 
    + 2\alpha^2V + \alpha^2 + \alpha^2U^2 + \alpha^2W^2T^2 + 2\alpha^2WUT + 2\alpha^2 U \nonumber \\
    + 2\alpha^2 WT + 2\alpha^2 VWT + 2\alpha^2 UV = q \label{p52}
    \\ UV=RT \label{p53}
\end{gather}
Using \eqref{p51} and \eqref{p52}, we can eliminate $U$ and $R$, and what we get is a single equation in three variables, of degree 3:
\begin{equation*}
    \alpha TVW + \alpha V^2 + T^2 + VW + \alpha TW + V(\alpha-p) + T(c+\alpha^{-1})=0,
\end{equation*}
where $c=\frac{q-p^2}{2}$. Applying some further linear substitutions $T=X-\alpha^{-1}$, $V=\alpha^{-1/2}Y-1$, and $W=\alpha^{-1/2}Z$ results in the equation of the Betti moduli space 
\begin{equation*}
    {\mathcal{M}}_B^{JKTV}=\{(X,Y,Z)\in\mathbb{C}^3| XYZ+ X^2+ Y^2 + c_1X+c_2Y+c_3Z+c_4=0 \},
\end{equation*}
where $c_i=c_i(p,q,\alpha)\in\mathbb{C}$ are some constants ($i=1,2,3,4$). 
\begin{remark}
    This equation is in agreement with the one in \cite[Subsection 3.2]{vPS}, for the JM representation of Painlev\'e V case. For a suitable choice of $c_i$, we can obtain the parameters $s_i$ in \cite{vPS}, and ${\mathcal{M}}_B^{JKTV}\rightarrow\mathcal{P}= \mathbb{C}\times \mathbb{C}\times\mathbb{C}^{\times}$ is indeed a family of affine cubic surfaces (with our notation the coordinates of the base are $p,q\in\mathbb{C}$ and $\alpha\in\mathbb{C}^{\times}$).
\end{remark}

\subsection{The $JKTIVa$ case}

Here we have a similar divisor as in the previous two cases, but now the irregular singularity is maximally twisted. That is, the local form at $c_1=\infty$ is \eqref{polar3}, and its exponential factors are given by \eqref{eq:qi_JKTIVa}:  $q_{0,1,2}=-a_0w^{-1}+\lambda_2w^{-2/3}+O(w^{-1/3})$. Therefore, formula \eqref{eq:realpart} here reads $\textrm{Re}(\lambda_2(\varepsilon^i-\varepsilon^j)w^{-2/3})=0$.
This time, we set $\textrm{Im}(\lambda_2)=0$ by Remark \ref{rmrk:Stokes}, and the solutions for the Stokes directions are $\varphi=k\frac{\pi}{2}$, while $\varphi=k\frac{\pi}{2}+\frac{\pi}{4}$ for the singular directions $k=1,2,3,4$. 
 Their order is given by $\{q_0-q_1\}$, $\{q_0-q_2\}$, $\{q_1-q_2\}$, and $\{q_1-q_0\}$. So, the four Stokes matrices are:
\begin{gather*}
    S_1= \begin{pmatrix}
1 & x_1 & 0\\
0 & 1 & 0\\
0 & 0 & 1
\end{pmatrix}, 
S_2= \begin{pmatrix}
1 & 0 & x_2\\
0 & 1 & 0\\
0 & 0 & 1
\end{pmatrix}, 
S_3= \begin{pmatrix}
1 & 0 & 0\\
0 & 1 & x_3\\
0 & 0 & 1
\end{pmatrix}, 
S_4= \begin{pmatrix}
1 & 0 & 0\\
x_4 & 1 & 0\\
0 & 0 & 1
\end{pmatrix}.
\end{gather*}
The topological monodromy at $\infty$ has the form $M_{\infty}=HS_4S_3S_2S_1$, and \eqref{conjugacy} translates to:
\begin{gather}
    \textrm{Tr}(M_{\infty})=p=x_1+x_3+x_2x_4 \label{p41} \\
    \textrm{Tr}(M_{\infty}^2)=q=2x_4+x_1^2+2x_2+2x_1x_2x_4+x_3^2+x_2^2x_4^2+2x_2x_3x_4 \label{p42}
\end{gather}
As we described, the group action does not decrease the dimension, the invariant monomials would be $x_1x_4$, $x_2$, $x_3$, $x_1x_3$, $x_2x_4$ (no $x_5,x_6$ here, see Remark \ref{rmrk:variables}), so the introduction of new variables would not affect the mondoromy data. Now expressing $x_1$ from \eqref{p41}, and substituting to \eqref{p42}, we obtain:
\begin{equation}
\label{eq:p4a}
    x_2x_3x_4+x_3^2+x_4-px_3+x_2+\frac{p^2-q}{2}=0
\end{equation}
In other words, we have the Betti moduli space:
\begin{equation*}
    {\mathcal{M}}_B^{JKTIVa}=\{(X,Y,Z)\in\mathbb{C}^3| XYZ+ X^2 + c_1X+c_2Y+c_3Z+c_4=0 \},
\end{equation*}
for suitable choice of parameters $c_i=c_i(p,q)$ ($i=1,2,3,4$).
\begin{remark}
    This Painlev\'e IV-type Betti moduli space is a family of affine cubic surfaces over parameter space $\mathcal{P}=\mathbb{C}\times\mathbb{C}$. There is an identification for the parameters $s_1,s_2$ in \cite[Subsection 3.7]{vPS} with the current ones, which we explain here. The equation of the monodromy space for the JMIV family in \cite{vPS} is:
    \[
    x_1x_2x_3+x_1^2-(s_2^2+s_1s_2)x_1-s_2^2x_2-s_2^2x_3+(s_2^2+s_1s_2^3)=0.
    \]
Introduce the new variables $x_1'=s_2^{-4/3}x_1$, $x_2'=s_2^{-2/3}x_2$, $x_3'=s_2^{-2/3}x_3$. Substitute them into the above equation and divide the equation by $s_2^{8/3}$ (recall that $s_2\in\mathbb{C}^{\times}$ in \cite{vPS}). This results in the equation 
\[
x_1'x_2'x_3'+x_1'^2-(s_2^2+s_1s_2)s_2^{-4/3}x_1'-x_2'-x_3'+(s_2^2+s_1s_2^3)s_2^{-8/3}=0.
\]
That is, the exact correspondence with \eqref{eq:p4a} is given by the change of variables $x_1'\mapsto x_3$, $x_2'\mapsto x_2$, $x_3'\mapsto x_4$, and the  relations between the parameters:
\begin{equation*}
    \begin{cases}
      (s_2^2+s_1s_2)s_2^{-4/3}=p \\
      (s_2^2+s_1s_2^3)s_2^{-8/3}=\frac{p^2-q}{2}.
    \end{cases}
\end{equation*}
\end{remark}

\subsection{The $JKTIVb$ case}

In the remaining three cases, we have only one singularity, which is irregular and is placed at $c_1=\infty$. Here, it is untwisted, which means that the local form is diagonal (see \eqref{polar4}), with exponential factors $q_i=-\frac{1}{2}a_iw^{-2}+O(w^{-1})$ from \eqref{eq:qi_JKTIVb}. Again, without loss of generality, we can choose $a_i=\varepsilon^i$, $i=0,1,2$ (see Remark \ref{rmrk:Stokes}). The solutions for \eqref{eq:sing} modulo $2\pi$ are $\varphi=k\frac{\pi}{6}+\frac{\pi}{12}$, $k=1,...,12$, that is, we have twelve singular directions, with Stokes matrices:
\begin{gather*}
    S_1= \begin{pmatrix}
1 & x_1 & 0\\
0 & 1 & 0\\
0 & 0 & 1
\end{pmatrix}, 
S_2= \begin{pmatrix}
1 & 0 & x_2\\
0 & 1 & 0\\
0 & 0 & 1
\end{pmatrix}, 
S_3= \begin{pmatrix}
1 & 0 & 0\\
0 & 1 & x_3\\
0 & 0 & 1
\end{pmatrix}, 
S_4= \begin{pmatrix}
1 & 0 & 0\\
x_4 & 1 & 0\\
0 & 0 & 1
\end{pmatrix}, 
\\ S_5= \begin{pmatrix}
1 & 0 & 0\\
0 & 1 & 0\\
x_5 & 0 & 1
\end{pmatrix}, 
S_6= \begin{pmatrix}
1 & 0 & 0\\
0 & 1 & 0\\
0 & x_6 & 1
\end{pmatrix},
S_7= \begin{pmatrix}
1 & x_7 & 0\\
0 & 1 & 0\\
0 & 0 & 1
\end{pmatrix}, 
S_8= \begin{pmatrix}
1 & 0 & x_8\\
0 & 1 & 0\\
0 & 0 & 1
\end{pmatrix},
\\S_9= \begin{pmatrix}
1 & 0 & 0\\
0 & 1 & x_9\\
0 & 0 & 1
\end{pmatrix}, 
S_{10}= \begin{pmatrix}
1 & 0 & 0\\
x_{10} & 1 & 0\\
0 & 0 & 1
\end{pmatrix},
S_{11}= \begin{pmatrix}
1 & 0 & 0\\
0 & 1 & 0\\
x_{11} & 0 & 1
\end{pmatrix}, 
S_{12}= \begin{pmatrix}
1 & 0 & 0\\
0 & 1 & 0\\
0 & x_{12} & 1
\end{pmatrix}.
\end{gather*}
The formal monodromy at $w=0$, is $H_1$, with $\alpha\beta\gamma=1$, similarly to the $JKTVI$ case.
The topological monodromy at $w=0$ is $M_{\infty}=H_1 \prod_{i=12}^1S_i$ equal to the identity, since we have only one singularity, see \eqref{identity}. Consider the matrix equation $S_6 S_5 S_4 S_3 S_2S_1=(H S_{12}S_{11}S_{10}S_9S_8S_7)^{-1}$:
\begin{gather*}
\begin{pmatrix}
1 & x_1 & x_2\\
x_4 &  x_1x_4+1 & x_3+x_2x_4\\
x_4x_6+x_5 & x_1x_4x_6+x_6+x_1x_5 & x_3x_6+x_2x_4x_6+x_2x_5+1
\end{pmatrix}    =
\\ = \begin{pmatrix}
\alpha^{-1}(x_7x_{10}+x_8x_{11}-x_7x_9x_{11}+1) & \beta^{-1}(-x_7+x_8x_{12}-x_7x_9x_{12}) & \gamma^{-1}(-x_8+x_7x_9)\\
\alpha^{-1}(-x_{10}+x_9x_{11}) & \beta^{-1}(x_9x_{12}+1) & \gamma^{-1}(-x_9)\\
\alpha^{-1}(-x_{11}) & \beta^{-1}(-x_{12}) & \gamma^{-1}
\end{pmatrix}
\end{gather*}
One can easily see that the variables $x_{11}, x_{10}$ can be expressed in terms of $x_1,...,x_6$, and also $x_9=-\gamma(x_3+x_2x_4)$, and $x_{12}=-\beta(x_1x_4x_6+x_6+x_1x_5)$. Since 
\begin{equation*}
    x_1=\beta^{-1}(-x_7+(x_8-x_7x_9)x_{12})=\beta^{-1}(-x_7-\gamma x_2x_{12})
\end{equation*}
we have $x_7=-\beta x_1-\gamma x_2x_{12}$, and similarly $x_8=-\gamma x_2+x_7x_9$. 
There remain to take into account the three equations for the diagonal entries. One of them is redundant, because of the condition on the determinant, for the other two we introduce the usual $G$-invariant monomials $U=x_1x_4$, $V=x_2x_5$, $W=x_3x_6$, $R=x_1x_3x_5$, $T=x_2x_4x_6$, with relation $UVW=RT$. The system describing the equation of the moduli space then reads as:
\begin{gather}
    \gamma W + \gamma T + \gamma V + \gamma = 1 \label{p31} \\
    U+1=\beta^{-1}(x_9x_{12}+1)=\beta^{-1}(\gamma(x_3+x_2x_4)\beta(x_1x_4x_6+x_6+x_1x_5)+1)= \nonumber \\
    = \gamma(UW+W+R+UT+T+UV)+\beta^{-1}\label{p32}
    \\ UVW=RT \label{p33}
\end{gather}
After eliminating $T$ and $R$ form \eqref{p31} and \eqref{p32}, and substituting to \eqref{p33}, we get the equation of the complex surface:
\begin{equation*}
    UVW+V^2+UV+UW+VW+U(1-\gamma^{-1})+V(2-\alpha-\gamma^{-1})+W(1-\alpha)+(1-\alpha)(1-\gamma^{-1})=0
\end{equation*}
With linear rescaling $U=X-1$, $V=Y-1$, $W=Z-1$ this turns into the form of the Painlev\'e IV-type Betti moduli space
\begin{equation*}
    {\mathcal{M}}_B^{JKTIVb}=\{(X,Y,Z)\in\mathbb{C}^3| XYZ+ Y^2 + c_1X+c_2Y+c_3Z+c_4=0 \},
\end{equation*}
where $c_1=-\gamma^{-1}$, $c_2=-\alpha-\gamma^{-1}-1$, $c_3=-\alpha$, and $c_4=\alpha\gamma^{-1}+\alpha+\gamma^{-1}$.
\begin{remark}
    Again we have ${\mathcal{M}}_B^{JKTIVb}\rightarrow\mathcal{P}=\mathbb{C}^{\times}\times\mathbb{C}^{\times}$, with cubic surfaces as fibers. Introduce new variables $X'=\alpha^{-1}X$, $Y'=\alpha^{-1}\gamma Y$, $Z'=\gamma Z$. By substituting them into the above equation and dividing it by $\alpha^2\gamma^{-2}$ (which is nonzero), we receive
    \[
    X'Y'Z'+Y'^2-\alpha^{-1}\gamma X'-(\alpha^{-1}\gamma+\gamma+\alpha^{-1})Y'-\alpha^{-1}\gamma Z'+(\alpha^{-2}\gamma+\alpha^{-1}\gamma^2+\alpha^{-1}\gamma)=0.
    \]
    Here $s_1=\alpha^{1/2}\gamma^{1/2}+\alpha^{-1/2}\gamma^{-1/2}$, $s_2=\alpha^{-1/2}\gamma^{1/2}$ provides the exact correspondence with the parameters $s_1,s_2$ in \cite[Subsection 3.7]{vPS}. 
\end{remark}

\subsection{The $JKTII$ case}

The irregular singularity at $c_1=\infty$ is minimally twisted in this case. This means that the leading order term of the local form has two Jordan blocks, one $1\times 1$ block with the exponential factor $q_2=-\frac{1}{2}a_1w^{-2}+O(w^{-1})$ and a $2\times 2$ block with the exponential factors $q_{0,1}=-\frac{1}{2}a_0w^{-2}+\lambda_3w^{-3/2}+O(w^{-1})$ (see formulas \eqref{polar5},\eqref{eq:qi_JKTII}). Therefore, one type of singular directions corresponds to pairs $\{q_0-q_1\}$, $\{q_1-q_0\}$, and $\{q_0-q_1\}$, equivalent to solving $\textrm{Im}(w^{-3/2})=0$, with phases $\varphi=0, \frac{2\pi}{3},\frac{4\pi}{3}$ in formula \eqref{eq:sing}. The other type of singular directions are for $\{q_2-q_{0,1}\}$, $\{q_{0,1}-q_2\}$, $\{q_2-q_{0,1}\}$, $\{q_{0,1}-q_2\}$, with phases $\varphi=\frac{\pi}{4},\frac{3\pi}{4},\frac{5\pi}{4},\frac{7\pi}{4}$, coming from \eqref{eq:sing} with appropriate choice of $a_0,a_1$. Thus, the Stokes matrices are in the appropriate order:
\begin{gather*}
    S_1= \begin{pmatrix}
1 & 0 & x_2\\
0 & 1 & x_3\\
0 & 0 & 1
\end{pmatrix}, 
S_2= \begin{pmatrix}
1 & x_1 & 0\\
0 & 1 & 0\\
0 & 0 & 1
\end{pmatrix}, 
S_3= \begin{pmatrix}
1 & 0 & 0\\
0 & 1 & 0\\
x_5 & x_6 & 1
\end{pmatrix}, 
S_4= \begin{pmatrix}
1 & 0 & 0\\
x_4 & 1 & 0\\
0 & 0 & 1
\end{pmatrix}, 
\\ S_5= \begin{pmatrix}
1 & 0 & x_8\\
0 & 1 & x_9\\
0 & 0 & 1
\end{pmatrix}, 
S_6= \begin{pmatrix}
1 & x_7 & 0\\
0 & 1 & 0\\
0 & 0 & 1
\end{pmatrix},
S_7= \begin{pmatrix}
1 & 0 & 0\\
0 & 1 & 0\\
x_{11} & x_{12} & 1
\end{pmatrix}.
\end{gather*}
Then for the topological monodromy $M_{\infty}=H\prod_{i=7}^1S_i=I$, with formal mondoromy $H_2$ at $z=\infty$. Consider the matrix equation: $S_3S_2S_1=(HS_7S_6S_5S_4)^{-1}$:
\begin{gather*}
    \begin{pmatrix}
1 & x_1 & x_2+x_3x_1\\
0 &  1 & x_3\\
x_5 & x_1x_5+x_6 & x_2x_5+x_3x_1x_5+x_3x_6+1
\end{pmatrix}=
\\ =\begin{pmatrix}
\alpha(x_7-x_8x_{12}) & x_8x_{11}+1 & \alpha^{-1}(-x_8)\\
\alpha(-x_4x_7+x_4x_8x_{12}-x_9x_{12}-1) &  -x_4-x_4x_8x_{11}+x_9x_{11} & \alpha^{-1}(x_4x_8-x_9)\\
\alpha x_{12} & -x_{11} & \alpha^{-1}
\end{pmatrix}.
\end{gather*}
We can express the variables $x_7,x_{12}$ in terms of $x_1,x_2,x_3x_5,x_6$ (by comparing the first and third elements of the first columns of the matrices), and also $x_{11}=-x_1x_5-x_6$, and $x_8=-\alpha(x_2+x_3x_1)$. Moreover
\begin{equation*}
    1=-x_4-(x_4x_8-x_9)x_{11}=-x_4-\alpha x_3x_{11},
\end{equation*}
so $x_4=-1-\alpha x_3x_{11}$, and because of $x_9=-\alpha x_3+x_4x_8$, the variables $x_4$ and $x_9$ can also be eliminated using the first five variables. We have not yet used the three relations for the diagonal entries, one of which is redundant because of the condition on the determinant, and again we gain one more relation $UVW=RT$ by introducing $U=x_2x_5$, $V=x_3x_6$, $W=x_1$, $R=x_2x_6$, $T=x_3x_1x_5$, $G$-invariant monomials. All in all, the system of equations that determine the moduli space is
\begin{gather}
    \alpha^{-1}=U+V+T+ 1 \label{p21} \\
    W=x_1=\alpha(x_2+x_3x_1)(x_1x_5+x_6)+1= \nonumber \\ =\alpha(UW+R+TW+VW)+1\label{p22}
    \\ UVW=RT \label{p23}
\end{gather}
Let us now express $T$ from \eqref{p21} and $R$ from \eqref{p22} and plug in \eqref{p23}, to get the equation:
\begin{equation*}
    UVW+UW+VW-\alpha^{-1}U-\alpha^{-1}V+W(1-\alpha^{-1})+\alpha^{-1}(\alpha^{-1}-1)=0
\end{equation*}
After the linear transformation of the variables $U=X-1$, $V=Y'-1$, $W=Z$, this equation transforms into:
\[
XY'Z -\alpha^{-1}X-\alpha^{-1}Y'-\alpha^{-1}Z+\alpha^{-1}(1+\alpha^{-1})=0
\]
Introducing $Y=\alpha Y'$ and factoring $\alpha$, the Betti moduli space is the following: 
\begin{equation*}
    {\mathcal{M}}_B^{JKTII}=\{(X,Y,Z)\in\mathbb{C}^3| XYZ - X - \alpha^{-1} Y - Z + 1+\alpha^{-1} = 0 \}.
\end{equation*}
\begin{remark}\label{rem:JKTII}
We obtained the same one-parameter family of affine cubic surfaces ${\mathcal{M}}_B^{JKTII}\rightarrow\mathcal{P}=\mathbb{C}^{\times}$ as in \cite[Subsection 3.9]{vPS}, up to the change of parameter $\alpha^{-1} \mapsto \alpha$. 
\end{remark}

\subsection{The $JKTI$ case}

The irregular singularity at $c_1=\infty$ is maximally twisted, the local form is \eqref{polar6}, and its exponential factors are $q_{0,1,2}=-\frac{1}{2}a_0w^{-2}+\lambda_5w^{-5/3}+O(w^{-4/3})$ from \eqref{eq:qi_JKTI}. 
The Stokes directions are equivalent to the phases $\varphi=k\frac{\pi}{5}$, $k=1,...,10$, which solve $\textrm{Re}((\varepsilon^i-\varepsilon^j)w^{-5/3})=0$ in \eqref{eq:realpart}, with the assumption $\textrm{Im}(\lambda_5)=0$ in Remark \ref{rmrk:Stokes}, while the singular directions are $\varphi=k\frac{\pi}{5}+\frac{\pi}{10}$, $k=1,...,10$. So, the Stokes matrices are in appropriate order:
\begin{gather*}
    S_1= \begin{pmatrix}
1 & x_1 & 0\\
0 & 1 & 0\\
0 & 0 & 1
\end{pmatrix}, 
S_2= \begin{pmatrix}
1 & 0 & x_2\\
0 & 1 & 0\\
0 & 0 & 1
\end{pmatrix}, 
S_3= \begin{pmatrix}
1 & 0 & 0\\
0 & 1 & x_3\\
0 & 0 & 1
\end{pmatrix}, 
S_4= \begin{pmatrix}
1 & 0 & 0\\
x_4 & 1 & 0\\
0 & 0 & 1
\end{pmatrix},
\\ S_5= \begin{pmatrix}
1 & 0 & 0\\
0 & 1 & 0\\
x_5 & 0 & 1
\end{pmatrix},
S_6= \begin{pmatrix}
1 & 0 & 0\\
0 & 1 & 0\\
0 & x_6 & 1
\end{pmatrix},
S_7= \begin{pmatrix}
1 & x_7 & 0\\
0 & 1 & 0\\
0 & 0 & 1
\end{pmatrix},
S_8= \begin{pmatrix}
1 & 0 & x_8\\
0 & 1 & 0\\
0 & 0 & 1
\end{pmatrix},
\\ S_9= \begin{pmatrix}
1 & 0 & 0\\
0 & 1 & x_9\\
0 & 0 & 1
\end{pmatrix},
S_{10}= \begin{pmatrix}
1 & 0 & 0\\
x_{10} & 1 & 0\\
0 & 0 & 1
\end{pmatrix}.
\end{gather*}
The topological monodromy at $\infty$ has the form $M_{\infty}=H_3\prod_{i=10}^1S_i$, and is equal to the identity $I$ by \eqref{identity}. Therefore the moduli space of this data is described by the variables $x_1,...,x_{10}$ under the condition $S_4S_3S_2S_1=(HS_{10}S_9S_8S_7S_6S_5)^{-1}$:
\begin{gather*}
    \begin{pmatrix}
1 & x_1 & x_2\\
x_4 &  x_1x_4+1 & x_3+x_2x_4\\
0 & 0 & 1
\end{pmatrix}=
\\ =\begin{pmatrix}
-x_8+x_7x_9 & x_7x_{10}+1 & -x_7\\
-x_9 &  -x_{10} & 1\\
x_6x_9+x_5x_8-x_5x_7x_9+1 & x_6x_{10}-x_5x_7x_{10}-x_5 & -x_6+x_5x_7
\end{pmatrix}.
\end{gather*}
Obviously $x_9=-x_4$, $x_{10}=-x_1x_4-1$, $x_7=-x_2$ can be expressed in terms of the first four variables, and also $x_8=x_7x_9-1=x_2x_4-1$. Remember, that here in the maximally twisted case, the maximal torus has dimension zero, so the introduction of the $G$-invariant monomials would not decrease the degree of freedom. Now
\begin{equation*}
    0=x_{10}(x_6-x_5x_7)-x_5=-x_{10}-x_5
\end{equation*}
because $1=-x_6+x_5x_7$, so $x_5=-x_{10}=1+x_1x_4$, moreover $x_6=-1+x_5x_7=-1-x_2(1+x_1x_4)$, so $x_5$ and $x_6$ can be eliminated as well. There are three further relations that we have not used for the four variables, but one of them is redundant, so we are left with two of them:
\begin{gather}
    x_3+x_2x_4=1 \label{p11} \\
    x_1=x_7x_{10}+1=x_1x_2x_4+x_2+1 \label{p12}
\end{gather}
From \eqref{p11} we can eliminate $x_3$, and from \eqref{p12}, with the change of variables: $x_1=-X$, $x_2=Y$, $x_4=-Z$, we get:
\begin{equation*}
    {\mathcal{M}}_B^{JKTI}=\{(X,Y,Z)\in\mathbb{C}^3| XYZ +X+Y+1=0 \}.
\end{equation*}
\begin{remark}
    This Painlev\'e I-type Betti moduli space is exactly same as the one in \cite[Subsection 3.10]{vPS}.
\end{remark}

\begin{center}
\begin{tabular}{ |c|c|c|  }
 \hline
 \multicolumn{3}{|c|}{Summary} \\
 \hline
 $JKT$ system& equation of $\mathcal{M}_B$ moduli space &dim$\mathcal{P}$\\
 \hline
 \hline
 $JKTVI$   & $\gamma XYZ+\alpha X^2+\beta Y^2 + \gamma Z^2 + c_1X+c_2Y+c_3Z+c_4=0$ &   4\\
 \hline
 $JKTV$ &  $XYZ+ X^2+ Y^2 + c_1X+c_2Y+c_3Z+c_4=0$ &3\\
 \hline
 $JKTIVa$ & $XYZ+ X^2 + c_1X+c_2Y+c_3Z+c_4=0$ &  2\\
 \hline
 $JKTIVb$&  $XYZ+ Y^2 + c_1X+c_2Y+c_3Z+c_4=0$  &   2\\
 \hline
 $JKTII$& $XYZ - X - \alpha^{-1} Y - Z + 1+\alpha^{-1} =0$   &1\\
 \hline
 $JKTI$& $XYZ +X+Y+1=0$  &0\\
 \hline
\end{tabular}
\captionof{table}{Summary of results on rank three character varieties, cf. the rank two JM representations of Painlev\'e cases \cite[Table 1.]{vPS}.}
\label{tab2}
\end{center}

\bigskip
\textbf{Funding}
\bigskip
\\ The project supported by the Doctoral Excellence Fellowship Programme (DCEP) is funded by the National Research Development and Innovation Fund of the Ministry of Culture and Innovation and the Budapest University of Technology and Economics, under a grant agreement with the National Research, Development and Innovation Office. During the preparation of this manuscript, the authors were supported by the grant K146401 of the National Research, Development and Innovation Office.
The second author was also supported by the grant KKP144148 of the National Research, Development and Innovation Office.

\end{document}